\documentclass{scrartcl}
\usepackage{amssymb,amsmath,amsthm}
\usepackage{dsfont}
\usepackage{float}
\usepackage{color}
\usepackage{url}
\usepackage{tikz}
\usetikzlibrary{automata,positioning,arrows}
\usepackage{authblk}
\newtheorem{Theorem}{Theorem}[section] 
\newtheorem{Lemma}[Theorem]{Lemma}     
\newtheorem{Corollary}[Theorem]{Corollary}
\newtheorem{Definition}[Theorem]{Definition}

\newtheorem{Example}[]{Example}
\newtheorem{Remark}[Theorem]{Remark}

\usepackage[square, sort, comma, numbers]{natbib} 

\renewcommand{\tilde}{\widetilde}

\newcommand{\FB}{\mathfrak{B}}

\newcommand{\St}{\mathrm{St}}
\newcommand{\CT}{\mathcal{T}}

\newcommand{\MA}{\mathds{A}}
\newcommand{\MM}{\mathds{M}}

\newcommand{\MF}{\mathds{F}}
\newcommand{\MN}{\mathds{N}}
\newcommand{\MD}{\mathds{D}}
\newcommand{\MV}{\mathds{V}}
\newcommand{\MZ}{\mathds{Z}}
\newcommand{\MQ}{\mathds{Q}}

\newcommand{\MR}{\mathds{R}}
\newcommand{\MC}{\mathds{C}}

\newcommand{\CO}{\mathcal{O}}
\newcommand{\CV}{\mathcal{V}}
\newcommand{\CF}{\mathcal{F}}
\newcommand{\CD}{\mathcal{D}}
\newcommand{\CX}{\mathcal{X}}
\newcommand{\CE}{\mathcal{E}}
\newcommand{\CP}{\mathcal{P}}

\newcommand{\MG}{\mathds{G}}

\newcommand{\FR}{\mathfrak{R}}
\newcommand{\fp}{\mathfrak{p}}

\newcommand{\fq}{\mathfrak{q}}
\newcommand{\UG}{\underline{G}}

\newcommand{\tphi}{\widetilde{\phi}}

\newcommand{\GL}{\mathrm{GL}}
\newcommand{\SL}{\mathrm{SL}}

\newcommand{\PSL}{\mathrm{PSL}}

 \newcommand{\SO}{\mathrm{SO}}
  
\newcommand{\stab}{\operatorname{Stab}} \newcommand{\diag}{\operatorname{diag}}

\newcommand{\End}{\mathrm{End}}

\newcommand{\Rep}{\mathrm{Rep}}
\newcommand{\Nred}{\mathrm{N}_{\text{red}}}
\newcommand{\dist}{\mathrm{dist}}

\newcommand{\Cl}{\mathcal{C}\!\ell}

\title{Computing $S$-unit groups of orders} 

 \author{Sebastian Sch\"onnenbeck\thanks{The author is supported by the DFG collaborative research center TRR 195} }
\affil{RWTH Aachen University\\
   Lehrstuhl D f\"ur Mathematik\\
   Pontdriesch 14/16, 52062 Aachen\\
   Germany\\
  \texttt{sebastian.schoennenbeck@rwth-aachen.de}
}

\begin{document}
\maketitle

\begin{abstract}
Based on the general strategy described by Borel and Serre and the Voronoi algorithm for computing unit groups of orders we present an algorithm for finding presentations of $S$-unit groups of orders. The algorithm is then used for some investigations concerning the congruence subgroup property.\end{abstract}

\section{Introduction}
Let $K$ be an algebraic number field with ring of integers $\CO_K$ and $S=\{\fp_1,...,\fp_s\}$ a finite set of maximal ideals of $\CO_K$. The ring of $S$-integers of $K$ is the set
\begin{equation}
 \CO_{K,S}:=\{x \in K~|~\nu_\fp(x) \geq 0 \text{ for all } \fp \notin S\},
\end{equation}
where $\nu_\fp$ denotes the discrete valuation at the ideal $\fp$. In other words, $\CO_{K,S}$ consists of all elements of $K$ that are locally integral at all primes not in $S$. By Dirichlet's unit theorem we know that its group of units $\CO_{K,S}^\times$ (also called the $S$-unit group of $K$) is - up to the group of roots of unity in $K$ - a free Abelian group of rank $a+s-1$ where $a$ is the number of pairwise inequivalent embeddings of $K$ into $\MC$.

The notion of $S$-units naturally generalizes to central simple algebras over $K$. Let $\MA$ be a finite dimensional, central simple $K$-algebra. Then by Wedderburn's theorem there is a skew field $\MD$ with center $K$ and some $n \in \MN$ such that $\MA \cong \MD^{n \times n}$.

For an $\CO_K$-order $\Lambda \subset \MA$ (i.e. a subring of $\MA$ that is a lattice over $\CO_K$) we set $\Lambda_S:=\Lambda \otimes_{\CO_K}\CO_{K,S}$ and call the group $\Lambda_S^\times$ the $S$-unit group of $\Lambda$. For groups of this type not much is known in general about their structure. In particular, there is no analogue of Dirichlet's unit theorem giving a uniform description. The aim of this article is to present a method for at least computing an explicit presentation (i.e. generators and defining relations) for groups of this type.

The groups in question are of interest for a variety of reasons: Firstly, the groups of type $\GL_n(\MD)$ are exactly the inner forms of $\GL_{nd}$ over $K$ where $d^2 = \dim_K(\MD)$ (see \cite[Prop. 2.17]{PlatonovRapinchuk}) and thus $S$-unit groups of orders are in some sense the most natural instances of $S$-arithmetic groups. 


 Furthermore, $S$-unit groups naturally appear as the unit groups of $S$-integral group rings $\CO_{K,S}[G]$ (where $G$ is a finite group). These groups play for instance a prominent role in the Zassenhaus conjecture (see \cite{KimmerleJahresbericht}) which states that any unit of finite order in $\MZ [G]$ is $\MQ[G]$-conjugate to $\pm g$ for some $g \in G$.

In addition, there are some open number theoretic conjectures concerning the groups of type $\Lambda_S^\times$. For a two-sided ideal $I \triangleleft \Lambda_S$ denote the group of all elements of $\Lambda_S^\times$ that are congruent to the identity modulo $I$ by $\Lambda_S^\times(I)$ and call this group the principal congruence subgroup of level $I$. Then $\Lambda_S^\times$ is said to have the congruence property if every finite index subgroup contains a principal congruence subgroup.
 In general the question whether $\Lambda_S^\times$ has the congruence property is widely open, however, it is conjectured (see \cite{PrasadRapinchuk}) that this is the case if the $S$-rank of $\GL_n(\MD)$ is at least $2$ and the $\fp$-rank of $\GL_n(\MD)$ is at least $1$ for all $\fp \in S$. 

As already stated, we want to describe an algorithm for computing an explicit presentation of $\Lambda_S^\times$. If $S=\emptyset$ there exists a general method for computing such presentations (see \cite{Brauncomputing}). On the other hand for $S \neq \emptyset$ explicit presentations are known only in a very limited number of cases. These cases either only deal with the Chevalley group $\SL_n$ over $K$ where $\CO_K$ is a principal ideal domain (see for instance \cite{BehrExplizite}) or only work for definite quaternion algebras over $\MQ$ (see \cite{Chinburgetal}) in which case the usual unit groups of orders are finite. The method we present here in principle works in the completely general setting. However, it is subject to rather strong constraints concerning computational power and memory. Due to these constraints we mainly focus our computational efforts (see Section \ref{ComputationalResults}) on the case $n=1$, i.e. computing $S$-unit groups of maximal orders in division algebras. To the best of the author's knowledge, the presentations given in this article are the first explicit presentations of $S$-arithmetic groups ($S \neq \emptyset$) in semi-simple algebraic groups with non-compact real points that are not of Chevalley type.


Our strategy for computing the presentation is as follows.
If $S=\emptyset$ the $S$-unit group of an order $\Lambda$ is simply its unit group $\Lambda^\times$ and in \cite{Brauncomputing} we already described an algorithm computing a presentation for this group. The algorithm employs the action of $\Lambda^\times$ on a certain cone of positive quadratic forms and, more importantly, also provides a framework to write arbitrary elements of $\Lambda^\times$ as a word in the computed generators.

For general $S$-arithmetic groups Borel and Serre (see \cite{BorelSerreCohomologie}) defined an action on a contractible CW-complex of the form
\begin{equation}
 \CX \times \prod_{i=1}^s \CX_{\fp_i}
\end{equation}
where $\CX$ is often called the Borel-Serre space (see \cite{RonanBuildings}) and $\CX_{\fp_i}$ is the Bruhat-Tits building of the corresponding algebraic group at the prime $\fp_i$. This action is then classically used to conclude (see \cite{BehrEndlich}) that $S$-arithmetic groups are finitely presented. 
For our approach we still employ this action, however, we do not consider the action on the whole complex but rather iteratively add one prime of $S$ after the other in an effort to only deal with a single component at a time. Moreover, in the first step, where we are dealing with an arithmetic group, we use the above-mentioned action on the cone of positive forms instead of the action on $\CX$. 


We aim to make the action on the Bruhat-Tits buildings as precise as possible in our special situation. Combining this action with the algorithm described above then gives rise to a practical method for computing generators and relations for $S$-unit groups of orders. 


The structure of the article is as follows. We start by briefly repeating the necessary details from Bass-Serre theory for computing presentations of groups acting on simply-connected complexes as well as the construction of a combinatorial model of the Bruhat-Tits building of the special linear group of a division algebra over a local field. Afterwards we make use of these preliminaries and describe the general strategy for computing a presentation for $S$-unit groups of orders and give some insight on how to solve the arising tasks in certain special cases.
In the last section we compute some examples where the algebra in question is a division algebra with center $\MQ$ or an imaginary quadratic field of class number $1$. We use the results for some investigations regarding the congruence subgroup property.
\section{Bass-Serre theory}
The method we want to employ for computing a presentation is based on the action of the group on a simply-connected CW-complex. The approach is generally known under the name Bass-Serre theory (see \cite{SerreArbres,BassCoveringTheory}). In our situation one has to be a little careful since the groups in question do not necessarily act orientation preservingly on the $1$-skeleton of the given CW-complex. To resolve this problem we follow the exposition in Brown's article (\cite{BrownPresentations}).

Let $G$ be a group and $X$ a simply-connected $G$-CW-complex, i.e. a simply-connected CW-complex on which $G$ acts by permuting cells.
Denote by $\CV$ the set of vertices of $X$ (cells of dimension $0$) and by $\CE$ the set of edges of $X$ (cells of dimension $1$). We fix an orientation for each $e \in \CE$, i.e. $e$ comes with two vertices $o(e)$ and $t(e)$ in $\CV$, the origin and target of $e$, respectively. Note that we can think of $(\CV,\CE)$ as a (directed) graph. For $e \in \CE$ we set $\overline{e}$ the same edge with reversed orientation (i.e. $o(\overline{e})=t(e)$ and $t(\overline{e})=o(e)$). We say that the orientation of $e$ is reversed by the action of $G$, if there exists $g \in G$ such that $e g =\overline{e}$; otherwise we say the orientation of $e$ is preserved by the action. Let us assume that we have chosen the orientation for the cells in $\CE$ in a way that $o(eg)=o(e)g$ and $t(eg)=t(e)g$ for all $g \in G$ whenever the orientation of $e$ is preserved by the action of $G$. Obviously this is always possible by fixing an orientation for a set of representatives of $\CE/G$ and then extending this by use of the $G$-action.

We decompose $\CE=\CE^+ \sqcup \CE^-$ where $\CE^+$ is precisely the set of edges whose orientation is preserved under the action of $G$. We say the edges are of plus- or minus-type, respectively. By 
\begin{equation}
  G_e=G_{o(e)} \cap G_{t(e)}                                                                                                                                                                                                                                                                                                        \end{equation}
 we denote the stabilizer of $e$ together with its orientation and by $G_{e,\overline{e}}=G_{\{o(e),t(e)\}}$ the stabilizer of $e$ ignoring the orientation. Clearly we have $[G_{e,\overline{e}}:G_{e}]$ equal to $1$ or $2$ according to $e \in \CE^+$ or $e \in \CE^-$.

We fix a tree $\CT$ in our graph such that the vertices $\CV_\CT$ of $\CT$ form a system of representatives of $\CV/G$. In particular, each edge of $\CT$ is in $\CE^+$. Choose systems $E^+,E^-$ of representatives of $\CE^+/G$ and $\CE^-/G$ such that $o(e) \in \CV_\CT$ for all $e\in E^+ \cup E^-$. For any $e \in E^+$ we choose an element $g_e \in G$ such that $t(e)g_e^{-1} \in \CV_\CT$ (with the convention $g_e=1$ if $t(e) \in \CV_\CT$). For any $e \in E^-$ we choose $g_e \in G_{e,\overline{e}}-G_{e}$.

Finally, we choose a system of representatives $\CF$ of the dimension $2$-cells of $X$ modulo the action of $G$  and fix for any $f \in \CF$ a sequence $(e_1,...,e_m)$ of edges with the following properties:
\begin{itemize}
 \item The $1$-cells in the boundary of $f$ are exactly $\{e_1,...,e_m\}$.
 \item $o(e_1) \in \CV_\CT$.
 \item $o(e_{i+1})=t(e_i)$ for $1 \leq i \leq m-1$ and $t(e_m)=o(e_1)$.
 \item $e_{i+1} \neq \overline{e_i}$ for all $ 1 \leq i \leq m-1$ and $e_1 \neq \overline{e_m}$.
\end{itemize}

To each $e_i,1 \leq i \leq m,$ we will non-canonically assign an element $g_i$, thought of as a word in the various $g_e$ and elements of the stabilizers $G_{e,\bar{e}}$ and $G_v$. The precise way to find these elements is described in \cite[Sect. 1]{BrownPresentations}. Here we merely note that we have $t(e_i)g_ig_{i-1}...g_1 \in \CV_\CT$ for all $1 \leq i \leq m$. Thus, in particular, $g_m...g_1 \in G_{o(e_1)}$. We call $(g_1,...,g_m)$ the cycle associated to $f$.

\begin{Theorem}[\protect{\cite[Thm. 1]{BrownPresentations}}]\label{Presentation}
 The group $G$ is generated by the $G_v,~v \in \CV_\CT,$ and the elements $g_e,~e \in E^+ \cup E^-,$ subject to the following relations:
 \begin{enumerate}
  \item The multiplication table of the groups $G_v,~v \in \CV_\CT$.
  \item $g_e=1$ if $e$ is an edge of $\CV_\CT$.
  \item $g_e \cdot g \cdot g_e^{-1} \in G_{t(e)g_e^{-1}}$ for all $e \in E^+$ and $g \in G_{e} \subset G_{o(e)}$.
  \item $g_e\cdot g \cdot g_e \in G_{o(e)}$ for all $e \in E^-$ and $g \in G_{e} \subset G_{o(e)}$.
  \item $g_m \cdot ...\cdot g_1 \in G_{o(e_1)}$ for any cycle $(g_1,...,g_m)$ associated to an element of $\CF$ as above.
 \end{enumerate}
\end{Theorem}

\section{Bruhat-Tits buildings}\label{Bruhat-Tits}
Let $F$ be a non-Archimedean local field of characteristic zero, $D$ an $F$-division algebra and $n \in \MN$. We quickly want to recall the construction of a combinatorial model for the affine Bruhat-Tits building $\FB$ of the reductive group $\SL_n(D)$ of norm-$1$-units in $D^{n \times n}$. A full account of this construction can be found in the book \cite{Garrett} and we note that similar constructions exist in the case of other classical groups (see \cite{AbramenkoNebe}) as well as for some exceptional groups (see \cite{GanYu1,GanYu2}).

We consider the simple right $D^{n \times n}$-module $V=D^{1 \times n}$ which we can also see as a left $D$-module. Let $\CO \subset D$ be the maximal order in $D$ and denote by $\pi \in \CO$ a generator of its maximal ideal. Then the vertices of $\FB$ are in natural bijection with the homothety classes
\begin{equation}
 [L]=\{\pi^k L~|~k \in \MZ\}
\end{equation}
 of left $\CO$-lattices in $V$. Moreover, a collection of vertices, say $[L_i],1 \leq i \leq s$, forms a simplex if and only if the union of the corresponding homothety classes, $\bigcup_{i=1}^s [L_i]$, forms a chain, i.e. is totally ordered by inclusion.

 Since the lattices between $\pi L$ and $L$ are in bijection with the subspaces of  $L/\pi L$, clearly a maximal simplex consists of exactly $n$ vertices (i.e. has dimension $n-1$). Furthermore we can construct such a maximal simplex $\{[L_i]~|~ 0 \leq i \leq n-1\}$ by choosing a basis $(b_1,...,b_n)$ of $V$ and defining
 \begin{equation}
  L_i=\bigoplus_{j=1}^{i} \CO \pi b_j \oplus \bigoplus_{j=i+1}^{n} \CO b_j \text{ for } 0 \leq i \leq n-1.
 \end{equation}

We say two (distinct) vertices, say $[L]$ and $[M]$, of $\FB$ are neighbours (or adjacent) if there is a simplex containing them both. Clearly this is the case if and only if $[L] \cup [M]$ forms a chain which gives rise to the following well-known lemma.

\begin{Lemma}
 Let $L$ be an $\CO$-lattice in $V$, then the neighbours of $[L]$ in $\FB$ are in natural bijection with the proper subspaces of $L/\pi L$.
\end{Lemma}

Finally, we note that the action of $\GL_n(D)$ on $\FB$ is just given by the usual action of $\GL_n(D)$ on the set of $\CO$-lattices in $V$.

\section{The general situation}\label{General Situation}
We now want to describe in some detail our general strategy for computing presentations of $S$-unit groups. In this section we try to describe the situation in as much generality as possible.

The key idea is that we will build the presentation iteratively, adding one prime of the set $S$ after the other. To that end we will assume that we can already handle $S$-arithmetic groups for a fixed set $S$ and are now working on $S'$-arithmetic groups, where we got $S'$ by adding one more prime to $S$.  This allows us to only deal with a single prime and thus in particular with a single Bruhat-Tits building at a time. 

Let $K$ be an algebraic number field with ring of integers $\CO_K$, $\MD$ a central $K$-division algebra and $n \in \MN$. We denote by $\MA$ the central simple $K$-algebra $\MA=\MD^{n \times n}=\mathrm{M}_n(\MD)$ and by $\MV \cong \MD^{1 \times n}$ its simple right module.

Let $\Lambda \subset \MA$ be a maximal $\CO_K$-order in $\MA$ and $\Delta \subset \MD$ a maximal $\CO_K$-order in $\MD$. Then there is a $\Delta$-left-lattice $L_0$ in $\MV$ such that $\Lambda=\End_{\Delta}(L_0)$ (cf. \cite[Cor. 27.6]{Reiner}).

Furthermore let $S$ be a finite set of prime ideals of $\CO_K$, $\fp \notin S$ another prime ideal and $S':=S \sqcup \{\fp\}$. We are interested in finding a presentation for the unit group $\Lambda_{S'}^\times =  (\Lambda \otimes_{\CO_K} \CO_{K,S'})^\times$ where $\CO_{K,S'}$ is the ring of $S'$-integers of $K$. 

We denote the completion of $K$ at $\fp$ by $K_\fp$ then there is $m \in \MN$ and a central $K_\fp$-division algebra $D_\fp$ such that $\MD_\fp=\MD \otimes_K K_\fp \cong D_\fp^{m \times m}$ and consequently $\MA_\fp = \MA \otimes_K K_\fp \cong D_\fp^{nm \times nm}$. Let $\CO_\fp$ be the maximal order in $D_\fp$ with prime element $\pi$. Since $\CO_\fp$ is a principal ideal domain we can without loss of generality assume that under the above isomorphism $\Delta_\fp:=\Delta \otimes_{\CO_K}\CO_{K_\fp}$ maps to $\CO_\fp^{m \times m}$ and we set 
\begin{equation}
 \epsilon:=\diag(\underbrace{1,0,...,0}_m) \in \Delta_\fp
\end{equation}
 (using this identification). Using this idempotent we then set $V_\fp:=\epsilon \MV_\fp \cong D_\fp^{1 \times nm}$ the simple $\MA_\fp$-right-module.

Let us now denote the Bruhat-Tits building of $\SL_{mn}(D_\fp)$ by $\FB$. In particular, following Section \ref{Bruhat-Tits}, the vertices of $\FB$ are given by
\begin{equation}
 \{[M] ~|~ M \subset V_\fp~\CO_\fp-\text{left-lattice} \}
\end{equation}
where $[M]$ denotes the homothety class of $M$, i.e.
\begin{equation}
 [M]=\{\pi^i M ~|~ i \in \MZ \}.
\end{equation}
\begin{Remark}
 There is a natural embedding $A^\times \hookrightarrow A_\fp^\times$ so $A^\times$ (and thus in particular $\Lambda_{S'}^\times$) acts on the set of $\CO_\fp$-lattices in $V_\fp$ and thus on the vertices of $\FB$ by right-multiplication.
\end{Remark}
The following theorem illustrates a well-known but very important principle.
\begin{Theorem}
 Let 
\begin{equation}
 \MM(L_0,S'):=\{M \subset \MV~|~ M~ \Delta_S\text{-left-lattice, }M\otimes \CO_{K_\fq} = L_0 \otimes \CO_{K_\fq} ~\forall~ \fq \notin S' \}.
\end{equation}
The map
\begin{equation}
 \begin{split}
  \phi:&\MM(L_0,S') \rightarrow \{M \subset V_\fp~|~ M ~\CO_\fp\text{-left-lattice} \}\\
   & M \mapsto \epsilon (M \otimes_{\CO_{K,S}} \CO_{K_\fp})
 \end{split}
\end{equation}
is an equivalence of $\Lambda_{S'}^\times$-sets, i.e. it is a bijection compatible with the respective actions of $\Lambda_{S'}^\times$ on the left- and right-hand-side.
\end{Theorem}
\begin{proof}
 Since $\Delta_S \otimes_{\CO_{K,S}} \CO_{K_\fp}\cong \CO_\fp^{m \times m}$ we can write $1_{\Delta_S}$ as a sum of orthogonal primitive idempotents, 
\begin{equation}
1_{\Delta_S}=\epsilon_1+...+\epsilon_m \in \Delta_S \otimes_{\CO_{K,S}} \CO_{K_\fp}\cong \CO_\fp^{m \times m}
\end{equation}
with $\epsilon_1=\epsilon$ and moreover $\epsilon_i=g_i^{-1}\epsilon g_i$ for some $g_i \in (\Delta_S \otimes_{\CO_{K,S}} \CO_{K_\fp})^{\times}$ and $2 \leq i \leq m$. Let us also set $g_1 = 1$ for the sake of uniformity.

 We will first show that $\phi$ is surjective. 
 Let $M \subset V_\fp = \epsilon \MV_\fp$ be an arbitrary $\CO_\fp$-lattice. Then $M_i:=g_i^{-1} M$ (where we actually think of $M$ as a subset of $\epsilon \MV_\fp$) is a full lattice in $\epsilon_i\MV_\fp$ and thus $\widehat{M}=M_1+...+M_m$ is a full lattice in $\MV_\fp$ with $\epsilon_1 \widehat{M} =M$. By the local-global principle for $\Delta_S$-lattices there now exists a $\Delta_S$-lattice $N \subset \MV$ such that $N \otimes \CO_{K_\fq} = L_0 \otimes \CO_{K_\fq}$ for $\fq \notin S'$ and $N \otimes \CO_{K_\fp} = \widehat{M} \otimes \CO_{K_\fp}$ which means that $N$ is a preimage of $M$ under $\phi$.

 Now let $L,M \in \MM(L_0,S')$. Then both $L \otimes_{\CO_{K,S}} \CO_{K_\fp}$ and $M \otimes_{\CO_{K,S}} \CO_{K_\fp}$ are $\Delta_S \otimes_{\CO_{K,S}} \CO_{K_\fp}$-lattices by construction.  We compute
\begin{equation}
 \begin{split}
  \epsilon_i (L \otimes_{\CO_{K,S}}\CO_{K_\fp}) &= g_i^{-1}\epsilon g_i (L \otimes_{\CO_{K,S}}\CO_{K_\fp}) \\
&=g_i^{-1}\epsilon(L \otimes_{\CO_{K,S}}\CO_{K_\fp}) \\
&=g_i^{-1} \phi(L) \\
&=g_i^{-1} \phi(M) \\
&=g_i^{-1} \epsilon g_i (M \otimes_{\CO_{K,S}}\CO_{K_\fp})
 \end{split}
\end{equation}
and thus
\begin{equation}
\begin{split}
 L \otimes_{\CO_{K,S}}\CO_{K_\fp}&=1_{\Delta_S}(L \otimes_{\CO_{K,S}}\CO_{K_\fp}) \\
&=(\epsilon_1+...+\epsilon_m)(L \otimes_{\CO_{K,S}}\CO_{K_\fp})\\
&=(\epsilon_1+...+\epsilon_m)(M \otimes_{\CO_{K,S}}\CO_{K_\fp})\\
&=M \otimes_{\CO_{K,S}}\CO_{K_\fp}.
\end{split}
\end{equation}
Since we also have $L\otimes_{\CO_{K,S}}\CO_{K_\fq}=M\otimes_{\CO_{K,S}}\CO_{K_\fq}$ for all $\fq \notin S'$, we hence obtain $L\otimes_{\CO_{K,S}}\CO_{K_\fq}=M\otimes_{\CO_{K,S}}\CO_{K_\fq}$ for all $\fq \notin S$, whence $L=M$ by the local-global principle for $\CO_{K,S}$-lattices. Thus $\phi$ is injective as well.

The compatibility with the $\Lambda_{S'}^\times$-action is obvious.
\end{proof}

\begin{Corollary}\label{LatticeStabilizer}
 Let $L \in \MM(L_0,S')$. Then
\begin{equation}
\stab_{\Lambda_S'^{\times}}(\phi(L))=(\End_{\Delta_S}(L))^\times = \{g \in \MA~|~ Lg=L\}.
\end{equation}
\end{Corollary}
\begin{proof}
 Since we have an equivalence of $\Lambda_{S'}^\times$-sets the only thing we need to show is that any $g \in \MA$ that fixes $L$ is already an element of $\Lambda_{S'}^\times$ but this is clear since $L$ and $L_0$ coincide away from $S'$.
\end{proof}

Now let $\CP$ be the ideal of $\Delta_S$ such that $\CP \otimes_{\CO_{K,S}}\CO_{K_\fp} = (\pi \CO_\fp)^{m \times m}$. Moreover for a lattice $L$ we set
\begin{equation}
 [L]=\{\CP^k L~|~k \in \MZ \}.
\end{equation}

\begin{Lemma}
 The map $\phi$ extends to
 \begin{equation}
 \begin{split}
  \tphi:&\{[M]~|~ M \in \MM(L_0,S')\} \rightarrow \text{Vertices of } \FB\\
   & [M] \mapsto [\phi(M)]
 \end{split}
\end{equation}
and this is again an equivalence of $\Lambda_{S'}^\times$-sets.
\end{Lemma}
\begin{proof}
 The only thing we need to show is that for $L,M \in \MM(L_0,S')$ we have $[L]=[M]$ if and only if $[\phi(L)] = [\phi(M)]$. But since $\epsilon$ and $\pi$ commute we see 
\begin{equation}
  \pi^i \epsilon(L \otimes \CO_{K_\fp})= \epsilon \pi^i(L \otimes \CO_{K_\fp})=\epsilon(\CP^i L \otimes \CO_{K_\fp})=\phi(\CP^i L). 
 \end{equation}
 The compatibility with the $\Lambda_{S'}^\times$-action is again clear.
\end{proof}

\begin{Corollary}\label{VertexStabilizer}
 Let $L$ be a lattice and $i \in \MZ_{> 0}$ minimal with the property that there exists $g \in \Lambda_{S'}^\times$ with $Lg=\CP^i L$ then
\begin{equation}
 \stab_{\Lambda_{S'}^\times}(\tphi([L])) = \stab_{\Lambda_{S'}^\times}([L]) = \langle (\End_{\Delta_S}(L))^\times,g \rangle.
\end{equation}
\end{Corollary}
\begin{proof}
The first equality is simply due to the fact that $\tphi$ is an equivalence of $\Lambda_{S'}^\times$-sets.
 Now let $i$ and $g$ be as in the assertion and let $h \in \stab_{\Lambda_{S'}^\times}([L])$. Then there is $j \in \MZ$ such that $Lh=\CP^j L$. We choose integers $a,b$ such that $ai+bj=\mathrm{gcd}(i,j)$ and see
\begin{equation} 
Lg^ah^b=\CP^{ai+bj}L=\CP^{\mathrm{gcd}(i,j)}L.
\end{equation}
 By the choice we made for $i$ we thus have $j =k\cdot i$ for some $k \in \MZ$ whence $L(hg^{-k})=L$. But this already implies $hg^{-k} \in \stab_{\Lambda_{S'}^\times}(L)$ and so 
\begin{equation}
h \in \langle \stab_{\Lambda_{S'}^\times}(L),g \rangle = \langle (\End_{\Delta_S}(L))^\times,g \rangle.
\end{equation}
\end{proof}

\begin{Remark}\label{StabilizerRemark}
\begin{enumerate}
 \item Following the above lemma there is a short exact sequence
\begin{equation}
 1 \rightarrow (\End_{\Delta_S}(L))^\times \hookrightarrow \stab_{\Lambda_{S'}^\times}([\phi(L)]) \twoheadrightarrow (\MZ,+) \rightarrow 0.
\end{equation}
Thus we obtain a presentation for $\stab_{\Lambda_{S'}^\times}([\phi(L)])$ if we have one for $(\End_{\Delta_S}(L))^\times$ and can perform constructive membership in the given generators.
\item There is a positive integer $k$ such that $\CP^k$ is generated by some central element $\rho \in \CO_{K,S}$. Then $L\cdot(\rho \cdot 1_{\Lambda_{S'}})=\rho L =\CP^kL$ thus the integer $i$ from the above lemma is bounded by $k$.
\end{enumerate}
\end{Remark}

The fact that $\tphi$ is an equivalence of $\Lambda_{S'}^\times$-sets also allows us to check whether to vertices of $\FB$ are in the same $\Lambda_{S'}^\times$-orbit.
\begin{Lemma}\label{LatticeIsomorphisms}
 Let $L,M$ be lattices and let $i \in \MZ_{>0}$ be minimal with the property that there exists $h \in \Lambda_{S'}^\times$ such that $M h =\CP^i M$. Then $[L]$ and $[M]$ (and equivalently $\tphi([L])$ and $\tphi([M])$) are in the same $\Lambda_{S'}^\times$-orbit if and only if there exists $0 \leq k < i$ such that there is an isomorphism of $\Lambda_S$-lattices between $L$ and $\CP^k M$.
\end{Lemma}
\begin{proof}
 First assume that there is an element $g \in \Lambda_{S'}^\times$ such that $[L]g=[M]$, so $Lg =\CP^j M$ for some $j \in \MZ$. Then there is an integer $a$ such that $0 \leq j+ai <i$ and we have $Lgh^a=\CP^{j+ai}M$ so $gh^a$ is an isomorphism of $\Lambda_S$-lattices between $L$ and one of the lattices $\CP^k M$ with $0 \leq k <i$.
 
 On the other hand if $L$ and $\CP^k M$ (with $0 \leq k < i$) are isomorphic as $\Lambda_S$-lattices there exists $g \in A^\times$ such that $Lg = \CP^k M$ and since $g$ fixes $L$ (and thus $L_0$) away from $S'$ it is an element of $\Lambda_{S'}^\times$.
\end{proof}
Let us denote by $\mathrm{Rep}(L_0,S,\fp)$ a set of lattices such that $\{[L]~|~L \in \mathrm{Rep}(L_0,S,\fp) \}$ is a system of representatives of $\{[L]~|~L \in \MM(L_0,S')\}$ modulo the action of $\Lambda_{S'}^\times$.
\begin{Corollary}
Since there are only finitely many isomorphism classes of $\Lambda_S$-lattices for a given dimension, the set $\mathrm{Rep}(L_0,S,\fp)$ is necessarily finite.
\end{Corollary}

The edges in $\FB$ are of the form 
\begin{equation}
 [L'] - [M']
\end{equation}
where $L'$ and $M'$ are $\CO_\fp$-lattices in $V_\fp$ and $\pi L' \subsetneq M' \subsetneq L'$ or equivalently the edges are of the form
\begin{equation}
 [\phi(L)] - [\phi(M)]
\end{equation}
where $L,M \in \MM(L_0,S')$ with $\CP L \subsetneq M \subsetneq L$. In particular, any vertex in $\FB$ has degree 
\begin{equation}
 \sum_{k=1}^{nm-1}{\binom{nm}{k}}_q
\end{equation}
where $q$ is the order of the residue field $\CO_\fp/\pi$ and ${\binom{nm}{k}}_q$ denotes the Gaussian binomial coefficient. The explicit description of the edges also allows us to quickly determine the combinatorial distance between two vertices of $\FB$ (i.e. the length of a shortest path between the two).
\begin{Lemma}
 Let $L, M \in \MM(L_0,S)$ and let $k \in \MZ$ be minimal with $\CP^k M \subset L$. Then the distance between the vertices $\tphi([L])$ and $\tphi([M])$ in $\FB$ is
\begin{equation}
 \dist(\tphi([L]),\tphi([M]))=\min_{i \in \MZ} \CP^i L \subset \CP^kM.
\end{equation}
\end{Lemma}
\begin{proof}
 Since $[M]=[\CP^k M]$ we may assume $k=0$. Set $d:=\min_{i \in \MZ} \CP^i L \subset \CP^kM$ and look at the chain
\begin{equation}
 L \supset M+\CP L \supset \CP L \supset \CP^2 L + M \supset \CP^2 L \supset... \supset \CP^d L + M = M \supset \CP^d L.
\end{equation}
Then this chain defines a path of length $d$ between $[L]$ and $[M]$,
\begin{equation}
 \tphi([L])-\tphi([\CP L + M])-...-\tphi([\CP^{d-1} L +M])-\tphi([M]),
\end{equation}
so the distance between $[L]$ and $[M]$ is at most $d$.

On the other hand let 
\begin{equation}
 L=N_0 \supset N_1 \supset ... \supset N_t=M
\end{equation}
be a chain of lattices such that
\begin{equation}
 \tphi([N_0])-\tphi([N_1])-...-\tphi([N_{t-1}])-\tphi([N_t])
\end{equation}
is a shortest path between $\tphi([L])$ and $\tphi([M])$. It suffices to show that $\CP N_i \nsubseteq N_{i+2}$ for all $0 \leq i \leq t-2$, since this implies that $M \nsubseteq \CP^{t-1}L$. Assume on the contrary that there is some $i$ such that $\CP N_{i} \subset N_{i+2}$, but then $\tphi([N_i])-\tphi([N_{i+2}])$ is also an edge of $\FB$ which means that there is a shorter path between $\tphi([L])$ and $\tphi([M])$ in contradiction to our choice. Hence the distance between these two points is at least $d$ which proves the assertion.
\end{proof}

 Following the description of the edges we find that all $2$-cells (i.e. triangles) are of the form 
  
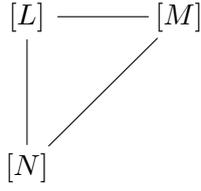
\begin{center}
\begin{tikzpicture}
  \draw[-, shorten <= .4cm, shorten >= 0.4cm,above] (0,2) to (2,2);
  \draw[-, shorten <= .3cm, shorten >= 0.3cm,left] (0,2) to (0,0);
  \draw[-, shorten <= .4cm, shorten >= 0.4cm,above] (2,2) to (0,0);
  \node at (0,0) {$[N]$};
  \node at (0,2) {$[L]$};
  \node at (2,2) {$[M]$};
\end{tikzpicture}
\end{center}
where $L,M,N \in \MM(L_0,S')$ such that $\CP L \subsetneq N \subsetneq M \subsetneq L$.

\begin{Lemma}\label{minustype}
 Let $L,M \in \MM(L_0,S')$ such that $[\phi(L)]=[\phi(Mg)]$ for some $g \in \Lambda_{S'}^\times$ and $E:[\phi(L)]-[\phi(M)]$ is an edge of $\FB$. Then $E$ is of the minus-type if and only if $[\phi(M)]$ is in the $\stab_{\Lambda_{S'}^\times}([\phi(L)])$-orbit of $Lg$. Moreover since $\stab_{\Lambda_{S'}^\times}([\phi(L)])$ is finitely generated and this orbit is finite this can be (constructively) decided by classical orbit enumeration. 
\end{Lemma}
\begin{proof}
 First note that for $E$ to be of minus-type the vertices $[\phi(L)]$ and $[\phi(M)]$ surely have to be in the same orbit under $\Lambda_{S'}^\times$.
 Now $E$ is of minus-type if there exists $h \in \Lambda_{S'}^\times$ such that $[\phi(Lh)]=[\phi(M)]$ and $[\phi(L)]=[\phi(Mh)]$. But then $[\phi(Mh)]=[\phi(Mg)]$ whence $hg^{-1}$ stabilizes $[\phi(M)]$ or equivalently $g^{-1} h = h^{-1} h g^{-1} h$ stabilizes $[\phi(Mh)]=[\phi(L)]$. This means that $[\phi(M)]=[\phi(Lh)]=[\phi(Lg)]g^{-1}h$ is indeed in the $\stab_{\Lambda_{S'}^\times}([\phi(L)])$-orbit of $Lg$.

 On the other hand if $[\phi(M)]=[\phi(Lg)]s$ for some $s \in \stab_{\Lambda_{S'}^\times}([\phi(L)])$ we set $h:=gs$ and compute 
 $[\phi(Mh)]=[\phi(Mg)]s=[\phi(L)]s=[\phi(L)]$ and $[\phi(Lh)]=[\phi(Lg)]s=[\phi(M)]$ so $E$ is of minus-type.
\end{proof}

After these preliminaries we are prepared to give a rough sketch of the general algorithm for obtaining a presentation of $\Lambda_{S'}^\times$:

\begin{enumerate}
 \item Starting from $L_0$, find a system of representatives $\FR:=\{[\phi(L_0)],...,[\phi(L_s)]\}$ of the vertices of $\FB$ modulo the action of $\Lambda_{S'}^\times$ by iteratively computing neighbours and checking for isomorphisms using Lemma \ref{LatticeIsomorphisms}.
 \item For each $0 \leq i \leq s$ compute a presentation for the stabilizer $\stab_{\Lambda_{S'}^\times}([\phi(L_i)])$ following Remark \ref{StabilizerRemark}.
 \item Compute a system of representatives of the edges of $\FB$ such that each representative has at least one vertex in $\FR$ (and both whenever possible) and compute for each representative $[\phi(L_i)]-[\phi(N)]$ an element $g_N \in \Lambda_{S'}^{\times}$ such that $[\phi(Ng_N)] \in \FR$ (ensuring that $g_N$ fixes the corresponding edge if it is of minus-type according to Lemma \ref{minustype}).
 \item According to \cite[Thm. 1]{BrownPresentations} the group $\Lambda_{S'}^\times$ is generated by the stabilizers of the $[\phi(L_i)]$ and the elements $g_N$ subject (only) to the following relations:
 \begin{enumerate}
  \item The relations among the generators of the stabilizers.
  \item The relations coming from the intersection of stabilizers of neighbouring vertices (differentiating between edges of plus- and minus-type).
  \item The relations arising from the $2$-cells.
 \end{enumerate}
\end{enumerate}

This provides an iterative approach to taking on the problem of finding a presentation for $\Lambda_{S'}^\times$, lowering the size of $S'$ by $1$ in each step. If $|S'|=1$ (or equivalently $S=\emptyset$) in step (2) we are in essence left with the task of finding a presentation for groups of the form $(\End_{\Lambda}(L))^\times$ which are unit groups of maximal orders and can thus be handled by the algorithm described in \cite{Brauncomputing}.

We now want to present one last results that significantly lowers the number of computations we actually have to perform when we apply the above algorithm. To that end let us denote the set of infinite places of $K$ by $\CV_\infty$.
\begin{Lemma}\label{StrongApproximation}
 Assume that $|\Lambda_{S}^\times/\CO_{K,S}^\times|=\infty$ and let $L \in \Rep(L_0,S,\fp)$. Then $G:=\stab_{\Lambda_{S'}^\times}(L)$ has exactly $nm-1$ orbits on the (undirected) edges of $\FB$ that contain $[L]$. More precisely there is a chain of $\Lambda_S$-lattices
 \begin{equation}
  \CP L \subsetneq M_1 \subsetneq M_2 \subsetneq ...\subsetneq M_{nm-1} \subsetneq L
 \end{equation}
 such that these orbits are precisely represented by $[L]-[M_i],~ 1 \leq i \leq nm-1$. 

 Moreover the orbits of $2$-cells containing $[L]$ under the action of $G$ are represented by
  
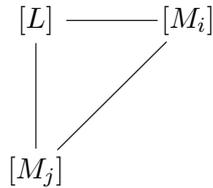
\begin{center}
\begin{tikzpicture}
  \draw[-, shorten <= .4cm, shorten >= 0.4cm,above] (0,2) to (2,2);
  \draw[-, shorten <= .3cm, shorten >= 0.3cm,left] (0,2) to (0,0);
  \draw[-, shorten <= .4cm, shorten >= 0.4cm,above] (2,2) to (0,0);
  \node at (0,0) {$[M_j]$};
  \node at (0,2) {$[L]$};
  \node at (2,2) {$[M_i]$};
\end{tikzpicture}
\end{center}
for $nm-1 \geq i >j \geq 1$.
\end{Lemma}
\begin{proof}
 After a change of basis we may assume that $\phi(L)=\CO_\fp^{nm} \subset V_\fp$. Since $\phi$ is a bijection we find lattices $M_1,...,M_{mn-1} \in \MM(L_0,S)$ such that 
 \begin{equation}
  \pi \phi(L) \subsetneq \phi(M_i)=\underbrace{\pi \CO_\fp \oplus ... \oplus \pi \CO_\fp}_{i} \oplus \underbrace{\CO_\fp \oplus ... \oplus \CO_\fp}_{nm-i} \subsetneq \phi(L).
 \end{equation}
Clearly the action of $G$ fixes the index $|L/M_i|$ so the edges $[L]-[M_i]$ are indeed in pairwise distinct orbits and the same holds for the given $2$-cells. 

On the other hand the group $\SL_{nm}(\CO_\fp/\pi)\cong \SL_{nm}(\MF_q)$ acts transitively on flags of subspaces of $\MF_\fq^{nm}$ (which in turn are in bijection with flags of sublattices between $\CP L$ and $L$). Hence it suffices to show that the image of the composition
\begin{equation}
 G \hookrightarrow \GL_{nm}(\CO_\fp) \rightarrow \GL_{nm}(\CO_\fp/\pi)
\end{equation}
 contains $\SL_{nm}(\CO_\fp/\pi)$, where the second homomorphism is just entry-wise reduction mod $\pi$ (see \cite[Thm. 4.3]{KleinertUnits} for the case $n=1$ and $K=\MQ$). To this end consider the norm-$1$-subgroup
 \begin{equation}
  \SL_n(\MD)=\{g \in \MD^{n \times n}~|~\Nred(g) =1 \}.
 \end{equation}
Then this defines an algebraic group over $K$ that is a form of the almost simple, simply-connected group $\SL_{nm'}$ where $(m')^2=\dim_K(\MD)$. Moreover, our condition on the order of $\Lambda_{S}^\times/\CO_{K,S}$ ensures that $\SL_n(\MD) \cap G$ is also infinite. But then (cf. \cite[Satz 2]{KneserStarkeApproximation}) $\SL_n(\MD)$ has the strong approximation property with respect to $S \cup \CV_\infty$ and thus the image of $(G \cap \SL_n(\MD)) \hookrightarrow \SL_{mn}(\CO_\fp)$ is dense which proves the assertion.
\end{proof}

\begin{Remark}
 The assumption in the previous lemma is automatically fulfilled if $n \geq 2$ or if $S \sqcup \CV_\infty$ contains a place at which $\MD$ is not totally ramified.
\end{Remark}

\section{Special cases}
In this section we want to describe how one can handle the individual tasks that arise in the algorithm described in the last section in certain special instances of algebras. We will use the same notation as in Section \ref{General Situation}.
\subsection{Matrix rings over number fields}
Let us first consider the case where $\MD=K$ is commutative. Then also $\Delta=\CO_K$ and without loss of generality we can assume $L_0=\CO_K^{n-1} \oplus I$ for some integral ideal $I \triangleleft \CO_K$, so
\begin{equation}
 \Lambda = \begin{pmatrix} \CO_K & \hdots &\CO_K & I \\
                            \vdots & \ddots &\vdots & \vdots \\
                            \CO_K & \hdots &\CO_K & I \\
                             I^{-1} & \hdots & I^{-1} & \CO_K
           \end{pmatrix}.
\end{equation}
Moreover, for $n \geq 2$ the condition of Lemma \ref{StrongApproximation} is already fulfilled for $S=\emptyset$ (and thus for every finite $S$).

As usual we denote by $\Cl(\CO_{K,S})$ the class group of $\CO_{K,S}$ and note that $\Cl(\CO_{K,S}) \cong \Cl(\CO_K)/\langle [\fq ]~|~\fq \in S \rangle$. Moreover, we will denote the Steinitz class of a $\CO_{K,S}$-lattice $L$ by $\St(L) \in \Cl(\CO_{K,S})$. 

\begin{Lemma}
 Let $L, M \in \MM(L_0,S)$.
\begin{enumerate}
 \item $L$ and $M$ are in the same $\Lambda_{S'}^\times$-orbit if and only if $\St(L) = \St(M)$.
 \item $[L]$ and $[M]$ are in the same $\Lambda_{S'}^\times$-orbit if and only if $\St(L)\langle [\fp^n] \rangle = \St(M)\langle [\fp^n] \rangle \in \Cl(\CO_{K,S})/\langle [\fp^n] \rangle$.
\end{enumerate}
\end{Lemma}
\begin{proof}
 \begin{enumerate}
  \item This is just the usual Steinitz theorem coupled with the fact that $L$ and $M$ are in the same orbit if and only if they are isomorphic as $\CO_{K,S}$-lattices.
  \item If there exists $g \in \Lambda_{S'}^\times$ such that $[L]g=[M]$, then there is an integer $k$ such that $Lg=\fp^k M$ and thus $\St(L)=\St(Lg)=[\fp^{kn}]\St(M)$. 

  On the other hand, if $\St(L)=\St(M)[\fp^{kn}]$ for some integer $k$ then $L$ and $\fp^k M$ are isomorphic as $\CO_{K,S}$-lattices and thus $[L]$ and $[\fp^k M]=[M]$ are in the same $\Lambda_{S'}^\times$-orbit.
 \end{enumerate}
\end{proof}
\begin{Remark}
 Constructively finding the isomorphisms in the above lemma can be done by finding pseudo bases for $L$ and $M$ in Steinitz form. The details can be found in \cite[Thm. 5.39]{PohstZassenhaus} and an implementation for example in \cite{Magma}. Furthermore there exists a generalization of these algorithms to the noncommutative case (see \cite[Alg. 2.2.6]{KirschmerHabil}).
\end{Remark}

\begin{Corollary}
 Let $k$ be the order of $[\fp]$ as an element of $\Cl(\CO_{K,S})$ and $t:=\mathrm{gcd}(k,n)$.
\begin{enumerate}
 \item $\Lambda_{S'}^\times$ has exactly $t$ orbits on the set $\{[L]~|~L \in \MM(L_0,S) \}$ and we can choose
\begin{equation}
\Rep(L_0,S,\fp)=\{[L_i]~|~0 \leq i < t \} \text{ where } L_i=\CO_{K,S} \otimes (\CO_{K}^{n-1} \oplus I\cdot \fp^i). 
\end{equation}
\item For each $0 \leq i < t$ there exists $g_i \in \Lambda_{S'}^\times$ such that $L_ig_i=\fp^{k/t}L_i$ and $\frac{k}{t}$ is the minimal positive integer with this property.
\end{enumerate}
 
\end{Corollary}

\subsection{Division algebras that satisfy the Eichler condition}
We now want to take a look at the case $n=1$ or equivalently $\MA=\MD$. For general division algebras $\MD$ the problems we have to solve in order to compute a presentation of $\Lambda_{S'}^\times$, for example the problem of constructively deciding whether two $\Delta_S$-lattices are isomorphic, might be hard. In particular, if $\MD$ is neither a field nor a quaternion algebra only a very limited number of (implemented) tools is available for dealing with $\Delta_S$-lattices. Here we want to describe a situation in which it is still reasonably easy to apply the algorithm from Section \ref{General Situation}.

Let
\begin{equation}
 \CV_{\infty,\MD}:=\{v \in \CV_{\infty}~|~ \MD \text{ ramifies at }v \}
\end{equation}
the set of all real places of $K$ that ramify in $\MD$,
\begin{equation}
 U(\MD):=\{a\in K~|~v(a)>0 \text{ for all } v \in \CV_{\infty,\MD} \}
\end{equation}
the set of elements of $K$ that are totally positive with respect to $\CV_{\infty,\MD}$ and finally
\begin{equation}
 \Cl_\MD(\CO_K):=\{ I \text{ fractional ideal of } \CO_K \}/\{a\CO_K~|~a \in U(\MD)\}
\end{equation}
the ray class group of $\CO_K$ with respect to $\CV_{\infty,\MD}$. We will assume that $\Cl_\MD(\CO_K)$ is trivial, i.e. that every fractional ideal of $\CO_K$ is generated by some element that is positive at all places at which $\MD$ ramifies. This is for example always the case (independent of $\MD$) if $K=\MQ$ or if $K$ is a CM-field and the usual class group, $\Cl(\CO_K)$, is trivial. Furthermore we assume that $\MD$ fulfills the Eichler condition, i.e. that $\MD$ is not a totally definite quaternion algebra.
The reason we make this assumption is the following theorem due to Eichler.
\begin{Theorem}[Eichler's theorem,\protect{\cite[Thm. 34.9]{Reiner}}]
If $\MD$ fulfills the Eichler condition, the reduced norm gives rise to a bijection between isomorphism classes of $\Delta$-left-ideals and the ray class group $\Cl_\MD(\CO_K)$.
\end{Theorem}
Since we assumed $\Cl_\MD(\CO_K)$ to be trivial we are thus in the situation that there is only one isomorphism class of $\Delta$-lattices in $\MV=\MD^{1 \times 1}$ and so the set $\Rep(L_0,S,\fp)$ only consists of the single element $L_0$ (in every iteration of the algorithm). Moreover, we can assume $L_0=\Delta=\Lambda$ without loss of generality.

Let us denote by $\lambda_\pi \in \Lambda$ an element such that $\Nred(\lambda_\pi) \in U(\MD)$ generates $\fp$. Then $\Delta_S \lambda_\pi$ is a maximal submodule of $\Delta_S$ (in particular of index $q^m$). 
\begin{Remark}
Following Lemma \ref{StrongApproximation} there are elements
\begin{equation}
 g_i \in \Lambda_S^\times,~1 \leq i \leq \frac{q^m-1}{q-1},
\end{equation}
such that any submodule of index $q^m$ is of the form $\Delta_S \lambda_\pi g_i$ for some $i$. Moreover, one can find the elements $g_i$ as explicit words in our chosen generators for $\Lambda_S^\times$ (even $\Lambda^\times)$ by a standard orbit computation.
\end{Remark}
\begin{Lemma}
 Any $\Delta_S$-left-ideal of index $q^{ma},~a \in \MZ_{>0}$, is of the form $\Delta_S g$, where $g \in \Lambda$ is a product of exactly $a$ elements of the set
\begin{equation}
 \left\{\lambda_\pi g_i,~1 \leq i \leq \frac{q^m-1}{q-1}\right\}.
\end{equation}
\end{Lemma}
Following this lemma we choose elements $w_1,...,w_{\left\lfloor \frac{m}{2} \right\rfloor} \in \Lambda_S$ (which we think of as explicit words in $\lambda_\pi$ and the generators of $\Lambda_S^\times$) such that 
\begin{equation}
 L_0 \supset L_0w_1 \supset L_0w_2...\supset L_0w_{\left\lfloor \frac{m}{2} \right\rfloor}
\end{equation}
 is a chain of submodules such that $|L_0/L_0w_i|=q^{mi}$ and $\CP L_0 \subset L_0w_i$ for all $i$. Moreover, we choose an element $s_\pi \in \Lambda_S$ such that $L_0 s_\pi =\CP L_0$ (note that $s_\pi$ is again a word in $\lambda_\pi$ and elements of $\Lambda_S^\times$).
\begin{Lemma}
 \begin{enumerate}
  \item The stabilizer of $[L_0]$ in $\Lambda_{S'}^\times$ is generated by $\Lambda_{S}^\times$ and $s_\pi$.
  \item The $\Lambda_{S'}^\times$-orbits on the edges of $\FB$ are represented by
\begin{equation}
 \tphi([L_0])-\tphi([L_0w_i]),~ 1 \leq i \leq \left\lfloor \frac{m}{2} \right\rfloor.
\end{equation}
\item The edge 
  \begin{equation}
   \tphi([L_0])-\tphi([L_0w_i])
  \end{equation}
 is of minus-type if and only if $i=\frac{m}{2}$ (in particular, this can only happen if $m$ is even).
\item There is a system of representatives of the $2$-cells modulo the action of $\Lambda_{S'}^\times$ corresponding to chains of the form
\begin{equation}
 L_0 \supset L_0w_i \supset L_0w_js_{i,j} w_i \supset \CP L_0
\end{equation}
with $0<i \leq j$, $ i \leq m-i-j$ and where $s_{i,j} \in \Lambda_S^\times$ is an arbitrary element with $L_0w_js_{i,j} \supset \CP L_0 w_i^{-1}$.
 \end{enumerate}
\end{Lemma}
 \begin{proof}
  \begin{enumerate}
   \item This is just Corollary \ref{VertexStabilizer}.
   \item Following Lemma \ref{StrongApproximation} the $\Lambda_S^\times$-orbits on the edges are represented by edges of the form
   \begin{equation}
    \tphi([L_0])-\tphi([N])
   \end{equation}
   where $\CP L_0 \subset N \subset L_0$ and we choose one such $N$ for each possible index $[L_0:N]=q^{am},~1 \leq a \leq m-1$. Now if $\CP L_0 \subset N \subset L_0$ with $[L_0:N]=q^{am}$ for some $ a>\left\lfloor \frac{m}{2} \right\rfloor$ there is some $g \in \Lambda_{S'}^\times$ such that $N=L_0g$ and we have
   \begin{equation}
    \begin{split}
    \tphi([L_0])-\tphi([N]) &= \tphi([\CP L_0])-\tphi([N])\\
    &= (\tphi([\CP L_0g^{-1}])-\tphi([L_0]))g \\
&= (\tphi([L_0])-\tphi([L_0g^{-1}]))g
    \end{split}
   \end{equation}
   where $[L_0:L_0g^{-1}]=q^{m(m-a)}$. Thus $\tphi([L_0])-\tphi([N])$ is in the $\Lambda_{S'}^\times$-orbit of $\tphi([L_0])-\tphi([L_0w_{m-a}])$. 

   On the other hand if $\CP L \subset N \subset L$ then the action of $\Lambda_{S'}^\times$ clearly fixes the set 
   \begin{equation}
    \{[L :N],[N:\CP L]\}
   \end{equation}
  so no two of the given edges are in the same orbit.
 \item Each edge of minus-type that contains $L_0$ is of the form $\tphi([L_0])-\tphi([N])$ where $\CP L_0 \subset N=L_0g \subset L_0$ and $g \in \Lambda_{S'}$ such that $[L_0g^2]=[L_0]$. But then necessarily $L_0g^2=\CP L_0$ and 
\begin{equation}
 [L_0:N]=[L_0g:L_0g^2]=[N: \CP L_0]=q^{m^2/2}.
\end{equation}
 Hence $\tphi([L_0])-\tphi([L_0w_{\frac{m}{2}}])$ is the only edge that can possibly be of minus-type.

 On the other hand the inclusions
\begin{equation}
 L_0w_{\frac{m}{2}} \supset \CP L_0 \supset \CP L_0w_{\frac{m}{2}} \text{ and }L_0w_{\frac{m}{2}} \supset  L_0w_{\frac{m}{2}}^2 \supset \CP L_0w_{\frac{m}{2}} 
\end{equation}
fulfill 
\begin{equation}
[L_0w_{\frac{m}{2}}:\CP L_0]=[L_0w_{\frac{m}{2}}:L_0w_{\frac{m}{2}}^2]=q^{m^2/2}.
\end{equation}
Thus the corresponding edges are in the same orbit under $\stab_{\Lambda_{S'}^\times}([L_0w_{\frac{m}{2}}])$ by Lemma \ref{StrongApproximation}. However, this already implies that $\tphi([L_0])-\tphi([L_0w_{\frac{m}{2}}])$ is of minus-type by Lemma \ref{minustype}.
\item Analogous to (2).
  \end{enumerate}
 \end{proof}

Following this lemma we assume in the following that $\Delta_S w_{m/2}^2 =\CP \Delta_S$ if $m$ is even. 
\begin{Remark}\label{CycleRemark}
 If
\begin{equation}
 L_0 \supset L_0w_i \supset L_0w_js_{i,j} \supset \CP L_0
\end{equation}
corresponds to a $2$-cell in the sense of part (4) of the above lemma there exists some $s_{i,j}' \in \Lambda_S^\times$ (not unique) such that
\begin{equation}
 \CP L_0 = L_0w_{m-i-j}s_{i,j}'w_js_{i,j}w_i
\end{equation}
and 
\begin{equation}
 w_{m-i-j}s_{i,j}'w_js_{i,j}w_i \in \stab_{\Lambda_{S'}^\times}([L_0])
\end{equation}
is the cycle corresponding to this $2$-cell in the sense of \cite{BrownPresentations}.
\end{Remark}

Following these preparations we are prepared to compute a presentation of $\Lambda_{S'}^\times$ fairly explicitly. Remember that $s_\pi$ as well as the $w_i$ are known to us as explicit words in $\lambda_\pi$ and elements of $\Lambda_S^\times$.
\begin{Lemma}
 The group $\Lambda_{S'}^\times$ is generated by $\Lambda_S^\times$ and $\lambda_\pi$ subject to the following relations:
\begin{enumerate}
 \item The relations in $\langle \Lambda_S^\times,s_\pi \rangle \cong \Lambda_S^\times \rtimes \MZ$.
 \item $w_i g w_i^{-1} \in \Lambda_S^\times$ for $g \in \Lambda_S^\times \cap (\Lambda_S^\times)^{w_i} \subset \Lambda_S^\times$ and $1 \leq i < \left\lceil \frac{m}{2} \right\rceil$.
 \item If $m$ is even: $w_{m/2} g w_{m/2} \in \langle \Lambda_S^\times,s_\pi \rangle$ for $g \in \Lambda_S^\times \cap (\Lambda_S^\times)^{w_{m/2}} \subset \Lambda_S^\times$.
 \item $w_{m-i-j}s_{i,j}'w_js_{i,j}w_i \in \stab_{\Lambda_{S'}^\times}([L_0])$ for $i\leq j, i \leq m-i-j$ and $s_{i,j},s_{i,j}'$ as above.
\end{enumerate}
\end{Lemma}

\section{Other algebraic groups}
Let $K$ be a number field, $\MG$ a reductive linear algebraic group over $K$ and $\UG$ an $\CO_K$-form of $\MG$. So far we were concerned with computing a presentation of an $S$-arithmetic subgroup $\UG(\CO_{K,S})$ of $\MG(K)$ in the case where $\MG$ is an inner form of $\GL_n$ over $K$. However, the general strategy we used, iteratively employing the action of $\UG(\CO_{K,S})$ on the Bruhat-Tits buildings, is in principle not limited to this case. As long as we have a workable model for the involved Bruhat-Tits buildings (which, as already mentioned, is the case for all classical and some exceptional groups) the algorithmic tasks one has to solve are in essence the same as for unit-groups of orders. In particular, one needs to decide whether two vertices are in the same orbit and, for the start of the iteration, one needs to be able to effectively compute with arithmetic subgroups of $\MG(K)$.

Constructively working with arithmetic subgroups in general is not easy. However, if $\MG(K \otimes_\MQ \MR)$ is compact, all arithmetic subgroups of $\MG(K)$ are finite and working with them becomes almost trivial (compared to the general situation). In this case the group $\MG$ clearly does not have the strong approximation property (with respect to $S=\emptyset$) so in general one has to deal with many orbits on edges and $2$-cells. However, in some rare situations the group $\UG(\CO_{K,\{\fp\}})$ acts transitively on the chambers of the Bruhat-Tits building (at $\fp$). If this is the case we are essentially in the same situation as if we had strong approximation and need only compute a reasonably small number of relations. These chamber-transitive groups were classified in \cite{KantorLieblerTits} and recently constructed using a different approach in \cite{KirschmerNebe}. 

We want to illustrate the general principle by giving an example in one of the exceptional groups.

\begin{Example}[\protect{\cite[Sect. 5.4]{KirschmerNebe}}]
 Let $\UG$ be the unique $\MZ$-form of $G_2$ such that $\UG(\MR)$ is compact and $\UG(\MZ_p)$ is a hyperspecial, maximal compact subgroup of $\UG(\MQ_p)$ for all primes $p$ (see \cite{GrossGroupsOverZ} for the construction). We consider the $S$-arithmetic group $G:=\UG(\MZ\left[\frac{1}{2}\right])$ (so $S=\{2\}$). It turns out (see for example \cite{KirschmerHabil}) that $G$ acts transitively on the chambers of the Bruhat-Tits building $\FB_2$ of $\UG(\MQ_2)$ and thus also on the vertices and edges of a given type in this building. The extended Dynkin-diagram of $G_2$ has the form 
\begin{center}
\begin{tikzpicture}

    \draw (2,0.07) -- (4,0.07);
    \draw (0,0) -- (4,0);
    \draw (2,-0.07) -- (4,-0.07);
    \draw (3.1,0) -- (2.9,0.2);
    \draw (3.1,0) -- (2.9,-0.2);
    \draw[fill=white] (0,0) circle(.1);
    \draw[fill=white] (2,0) circle(.1);
    \draw[fill=white] (4,0) circle(.1);
    
    \node at (-1,0) {$\tilde{G}_{2}:$};
    \node at (0,0.35) {$0$};
    \node at (2,0.35) {$1$};
    \node at (4,0.35) {$2$};

\end{tikzpicture} 
\end{center}
We choose vertices $v_0,v_1,v_2$ in $\FB_2$ of type $0,1$ and $2$, respectively, such that $\{v_0,v_1,v_2\}$ is a chamber. Since the group $G$ acts transitively on the edges (of a given type) it is already generated by the three stabilizers $\stab_G(v_i),~0\leq i \leq 2$. Since $G$ is also transitive on chambers there is only one $2$-cell one has to consider in the context of Brown's algorithm \ref{Presentation}, namely $C:=\{v_0,v_1,v_2\}$ which gives rise to the trivial relation. Thus $G$ is generated by the three groups
\begin{equation}
 \begin{split}
  \stab_G(v_0)&\cong G_2(2) \text{ of order } 2^6\cdot 3^3\cdot 7,\\
  \stab_G(v_1)&\cong 2_+^{1+4}.((C_3 \times C_3).2) \text{ of order } 2^6\cdot 3^2, \text{ and }\\
  \stab_G(v_2)&\cong 2^3.\GL_3(2) \text{ of order } 2^6\cdot 3\cdot 7\\
 \end{split}
\end{equation}
subject only to the relations arising from the fact that they intersect non-trivially. After some manual simplification we find that $G$ is generated by three elements $x_1,x_2,x_3$ of order $3$ which, under a suitable embedding $G_2 \hookrightarrow \SO_7$,  are given as

\begin{gather*}
 x_1=\frac{1}{2} \begin{pmatrix} -1&1&1 &-1&0&0&0\\
-1&0&0&1&0&1&1\\
-1&0 &-1&0&1& -1&0\\
 1&1&0&0&1&0&1\\
 0&0 &-1 &-1 &-1&0&1\\
 0 &-1&1&0&0 &-1&1\\
 0&1&0&1& -1& -1&0\end{pmatrix}, 
x_2 = \frac{1}{2} \begin{pmatrix} 0 & 1&1&0&0 &-1&1\\
 0&0 &-1& -1&1&0&1\\
-1&1 &-1&1&0&0&0\\
 1&0 &-1&0 &-1& -1&0\\
 1&1&0&0&1&0 &-1\\
 0&1&0 &-1 &-1&1&0\\
 1&0&0&1&0&1&1 \end{pmatrix},\\
x_3=\frac{1}{2}\begin{pmatrix}
      0 & 1&0&0&1&1& -1\\
-1&0&0&1&0&1&1\\
 1&0&0& -1&0&1&1\\
 1&1&0&1& -1&0&0\\
 0&0&2&0&0&0&0\\
 1 &-1&0&1&1&0&0\\
 0&1&0&0&1& -1&1
     \end{pmatrix}
\end{gather*}
 Each of these elements stabilizes one of the $1$-cells in the boundary of $C$ and transitively permutes the chambers containing this $1$-cell. 

 There are too many relations for it to make sense to print the full presentation here, so we provide a Magma readable version on the author's homepage instead:
{\normalfont
\begin{center}
 \url{www.math.rwth-aachen.de/homes/Sebastian.Schoennenbeck/S_unit_groups}
\end{center}
}
 The presentation can for instance be used in combination with Magma to verify that $G$ indeed has no normal subgroup of index at most $500,000$. 
\end{Example}
 
\section{Computational results}\label{ComputationalResults}
Due to the very limited usability of printed presentations and the fact that our algorithms generally yield rather long relations we refrain from actually printing the results here. Instead the results in a number of cases can be found on the author's homepage:

\begin{center}
 \url{www.math.rwth-aachen.de/homes/Sebastian.Schoennenbeck/S_unit_groups}
\end{center}

The available presentations include in particular those used in the following section to investigate the congruence subgroup property.
\section{The congruence subgroup property}\label{CSP}
We want to employ our algorithms for some experimental investigations regarding the congruence subgroup property. To that end we refer back to the notation of Section \ref{General Situation}. In particular, $\Lambda$ is the maximal order whose group of $S$-units we are interested in.
\begin{Definition}
 \begin{enumerate}
  \item For a two-sided ideal $I \triangleleft \Lambda_S$ we set $\Lambda_S^\times(I)$ the group of all elements of $\Lambda_S^\times$ that are congruent to the identity modulo $I$.
  \item The group $\Lambda_S^\times$ is said to have the congruence subgroup property if for every finite index subgroup $H \leq \Lambda_S^\times$ there is a two-sided ideal $I \triangleleft \Lambda_S$ such that $H$ contains $\Lambda_S^\times(I)$.
 \end{enumerate}
\begin{Remark}
 Set $\overline{S}:=S \cup \CV_\infty$ and let $\MG$ be the algebraic group (over $K$) arising from the norm-$1$-elements in $\MA^\times$. A conjecture of Serre (see \cite{PrasadRapinchuk}) states that $\Lambda_S^\times$ has the congruence subgroup property if
\begin{equation}
 \mathrm{rk_{\overline{S}}}\MG :=\sum_{v \in \overline{S}} \mathrm{rk_{k_v}}\MG \geq 2
\text{ and }\mathrm{rk_{k_\fp}} \MG >0\text{ for all }\fp \in S.
\end{equation} 
\end{Remark}
\end{Definition}

For an overview over the congruence subgroup property and related results we refer the interested reader to \cite{PrasadRapinchuk}.

We employ the strategy that was already used in \cite{Chinburgetal} for $S$-unit-groups in definite quaternion algebras and investigate the congruence subgroup property as follows. For certain instances of $\MA$ and $S$ we use our algorithm to compute a presentation for the projective $S$-unit group $\Lambda_S^\times/\CO_{K,S}^\times$. We then use Magma to compute all of its normal subgroups up to a modest index $n$ and check the composition factors that appear in the quotients. If $\Lambda_S^\times$ has the congruence subgroup property, the only non-Abelian simple groups that can appear here are of the form $\PSL_k(\MF_q)$ where $q$ is a power of $N(\fp)$ for some prime ideal $\fp$ of $\CO_{K,S}$ and
\begin{equation}
 \MA \otimes K_\fp \cong D_\fp^{ k \times k}.
\end{equation}

On the other hand for $\CP$ a prime ideal of $\Lambda_S$ we have $\Lambda_S^{\times}(\CP^2) \subset \Lambda_S^{\times}(\CP)$ and the quotient
\begin{equation}
 \Lambda_S^{\times}(\CP) / \Lambda_S^{\times}(\CP^2)
\end{equation}
 is a $p$-group, where $p\MZ = \CP \cap \MZ$. In particular, assuming the congruence subgroup property, the only constraints on the Abelian composition factors arise from the fact that the studied index has to be large enough for larger primes to appear. 
 
All of our computations support the congruence subgroup conjecture and we thus only tabulate the indices up to which we were able to compute all normal subgroups. We performed the computations for the following three algebras:

 Table \ref{D23} contains the indices for the algebra $\CD_{2,3}$, a degree $3$ algebra over $\MQ$ ramified at the primes $2$ and $3$. For all given instances of $S$ we only found Abelian composition factors of orders $2,3,7$ and $13$. These arise from the congruence subgroups corresponding to the ramified primes. In particular, note that $7 \mid 2^3-1$ and $13 \mid 3^3-1$. Table $\ref{Q57}$ contains the results for the rational quaternion algebra $Q_{5,7}$ ramified at $5$ and $7$. We find Abelian composition factors of orders $2,3,5$ and $7$ which all already arise at the ramified primes and non-Abelian composition factors isomorphic to $\PSL_2(q)$ with $q \in \{11,13,17,19\}$ (depending on $S$ and the index up to which we were able to compute) which arise from the congruence subgroups at the corresponding prime ideals of $\MZ$. Finally, Table $\ref{Quat-7}$ contains the results for the quaternion algebra $\left(\frac{-1,-1}{\MQ(\sqrt{-7})}\right)$. The two ideals that ramify in this algebra have norm $2$ and we obtain Abelian composition factors of order $2$ and $3$ as well as non-Abelian composition factors isomorphic to $\PSL_2(q)$ with $q \in \{7,9,11\}$ (depending on $S$ and the index) as one would expect.

\def \arraystretch{1.2}
\begin{table}[ht]
\begin{center}
 \begin{tabular}{l|l|l|l|l|l|l}

$S$ & $\{5\}$ &  $\{7\}$ & $\{11\}$& $\{5,7\}$ & $\{5,11\}$ & $\{5,7,11\}$ \\
\hline
$n$ & $10000$&  $10000$&  $2500$ &  $2500$&  $2500$&  $1000$
 \end{tabular}
\end{center}
\caption{Indices for computational check of the congruence subgroup property for $S$-unit groups in $\CD_{2,3}$.}\label{D23}
\end{table}


\def \arraystretch{1.2}
\begin{table}[ht]
\begin{center}
 \begin{tabular}{l|l|l|l|l|l|l}
$S$ &$\{2\}$ &$\{3\}$ &$\{11\}$ &$\{2,3\}$ & $\{2,11\}$& $\{2,3,11\}$\\
\hline
$n$ &$8000$ &$8000$ &$3000$ &$5000$ & $4000$& $3000$
 \end{tabular}
\end{center}
\caption{Indices for computational check of the congruence subgroup property for $S$-unit groups in $Q_{5,7}$.}\label{Q57}
\end{table}


\def \arraystretch{1.2}
\begin{table}[H]
\begin{center}
 \begin{tabular}{l|l|l|l|l|l|l}

$S$ &$\{\omega\}$ &$\{3\}$ & $\{2\omega-3\}$& $\{\omega,3\}$&$\{\omega,2\omega-3\}$ &$\{\omega,3,2\omega-3\}$ \\
\hline
$n$ &$5000$ &$5000$ &$5000$ &$5000$ &$5000$ &$2500$
 \end{tabular}
\end{center}
\caption{Indices for computational check of the congruence subgroup property for $S$-unit groups in $\left(\frac{-1,-1}{\MQ(\omega)}\right)$ where $\omega^2 = -7$.}\label{Quat-7}
\end{table}

\section*{Acknowledgements}
 The author would like to thank Professors Gabriele Nebe and Renaud Coulangeon for many helpful discussions as well as their comments and suggestions on earlier version of the article. 

\bibliographystyle{abbrv}
\bibliography{S-units}

\end{document}